\documentclass[11pt]{amsart}
\usepackage[foot]{amsaddr}
\usepackage{amsmath,amssymb,bm, graphicx,float}

\begin{document}
\newtheorem{theorem}{Theorem}[section]
\newtheorem{lemma}[theorem]{Lemma}
\newtheorem{corollary}[theorem]{Corollary}
\newtheorem{prop}[theorem]{Proposition}
\newtheorem{definition}[theorem]{Definition}
\newtheorem{remark}[theorem]{Remark}

 \def\ad#1{\begin{aligned}#1\end{aligned}}  \def\b#1{{\bf #1}} \def\hb#1{\hat{\bf #1}}
\def\a#1{\begin{align*}#1\end{align*}} \def\an#1{\begin{align}#1\end{align}}
\def\e#1{\begin{equation}#1\end{equation}} \def\t#1{\hbox{\rm{#1}}}
\def\dt#1{\left|\begin{matrix}#1\end{matrix}\right|}
\def\p#1{\begin{pmatrix}#1\end{pmatrix}} \def\c{\operatorname{curl}}
 \def\vc{\nabla\times } \numberwithin{equation}{section}
 \def\la{\circle*{0.25}}
\def\boxit#1{\vbox{\hrule height1pt \hbox{\vrule width1pt\kern1pt
     #1\kern1pt\vrule width1pt}\hrule height1pt }}
 \def\lab#1{\boxit{\small #1}\label{#1}}
  \def\mref#1{\boxit{\small #1}\ref{#1}}
 \def\meqref#1{\boxit{\small #1}\eqref{#1}}
\long\def\comment#1{}

\def\lab#1{\label{#1}} \def\mref#1{\ref{#1}} \def\meqref#1{\eqref{#1}}

\def\bg#1{{\pmb #1}} 

\title  [$C^1$-$P_k$ finite element]
   {Two families of C1-Pk Fraeijs de Veubeke-Sander finite elements on quadrilateral meshes}

\author { Shangyou Zhang }
\address{Department of Mathematical  Sciences, University of Delaware, Newark, DE 19716, USA.}
\email{ szhang@udel.edu }

\date{}

\begin{abstract}
We extend the $C^1$-$P_3$ Fraeijs de Veubeke-Sander finite element to
 two families of $C^1$-$P_k$ ($k>3$) macro finite elements on
  general quadrilateral meshes.
On each quadrilateral, four $P_k$ polynomials are defined on
  the four triangles subdivided from the quadrilateral by 
  its two diagonal lines.
The first family of $C^1$-$P_k$ finite elements is the full $C^1$-$P_k$
  space on the macro-mesh. 
Thus the element can be applied to interface problems.
The second family of $C^1$-$P_k$ finite elements condenses  
  all internal degrees of freedom by moving them to the four edges.
Thus the second element method has much less unknowns but is
  more accurate than the first one.
We prove the uni-solvency and the optimal order convergence.
Numerical tests and comparisons with the $C^1$-$P_k$ Argyris are provided.
\end{abstract}

\vskip .3cm



\maketitle

\section{Introduction} 
The finite elements were initially constructed for solving
     second order partial differential equations.
But the finite element methods became popular mainly because engineers 
   in early days started the constructions for the 4th order differential
  equation, the plate bending equation.  We can list some of these early constructions,  
 the $C^1$-$P_3$ Hsieh-Clough-Tocher element (1961,1965) \cite{Ciarlet,Clough}, 
  the $C^1$-$P_3$ Fraeijs de Veubeke-Sander element (1964,1965) \cite{Fraeijs,Fraeijs68,Sander}  
 the $C^1$-$P_5$  Argyris element (1968) \cite{Argyris},
  the $C^1$-$P_4$ Bell element (1969) \cite{Bell},
   the $C^1$-$Q_3$ Bogner-Fox-Schmit element (1965) \cite{Bogner},  and
 the $P_2$ nonconforming Morley  element (1969) \cite{morley}.
 
The $C^1$-$P_3$ Hsieh-Clough-Tocher element was extended to the $C^1$-$P_k$ ($k\ge 3$) 
   finite elements in \cite{Douglas,ZhangMG}.
The $C^1$-$P_5$  Argyris element was extended to the family of 
   $C^1$-$P_k$ ($k\ge 5$) finite elements in \cite{Zen70,Zlamal}. 
The $C^1$-$P_5$  Argyris element was modified and extended to the family of  
   $C^1$-$P_k$ ($k\ge 5$) full-space finite elements in \cite{Morgan-Scott}.
The $C^1$-$P_5$  Argyris element was also extended to 3D $C^1$-$P_k$ ($k\ge 9$)
   elements on tetrahedral meshes in \cite{Zenisek,Z3d,Z4d}.
The $C^1$-$P_4$ Bell element was extended to three families of 
    $C^1$-$P_{2m+1}$ ($m\ge 3$) finite elements in \cite{Xu-Zhang7,Xu-Zhang}. 
The Bell finite elements do not have any degrees of freedom on edges. 
Thus they must be odd-degree polynomials 
  (the $P_4$ Bell element is a subspace of $P_5$ polynomials.)
The $C^1$-$Q_3$ Bogner-Fox-Schmit element was extended to three families of $C^1$-$Q_k$ ($k\ge 3$)
  finite elements on rectangular meshes in \cite{Zhang-C1Q}.

In this work, we extend the $C^1$-$P_3$ Fraeijs de Veubeke-Sander element to two families
  of $C^1$-$P_k$ ($k\ge 3$) finite elements.
This is not done before.  But the $C^1$-$P_3$ Fraeijs de Veubeke-Sander element 
  was extended to $C^r$-$P_{d_r}$ ($r\ge 1$) finite elements in several publications
   \cite{Lag,Lag2,Lai,Moitinho}.

%

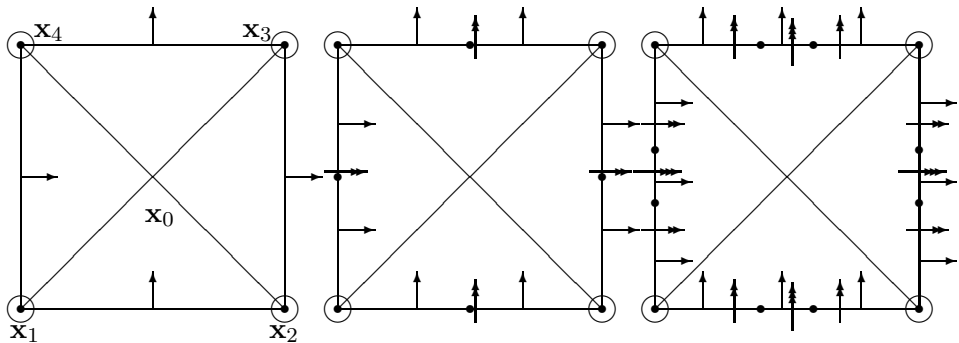
\begin{figure}[H]\centering
 \begin{picture}(360,120)(-10,-10) 
 \def\c{\circle*{3}}\def\h{\vector(1,0){14}}\def\v{\vector(0,1){14}}
 \def\mq{\begin{picture}(100,108)(0,0)  
  \put(0,0){\line(1,0){100}}  \put(0,0){\line(0,1){100}}  \put(100,0){\line(-1,1){100}} 
  \put(100,100){\line(-1,0){100}}  \put(100,100){\line(0,-1){100}}  \put(100,100){\line(-1,-1){100}} 
   \put(100,0){\circle*{3}} \put(100,0){\circle{10}}\put(100,100){\circle*{3}} \put(100,100){\circle{10}}  
    \put(0,0){\circle{10}} \put(0,0){\circle*{3}} \put(0,100){\circle*{3}} \put(0,100){\circle{10}}  
   \end{picture} }

 \put(0,0){\begin{picture}(100,108)(0,0)  \put(0,0){\mq} 
    \put(100,50){\h}   \put(50,100){\v}  \put(0,50){\h}     \put(50,0){\v}  
   \put(-4,-11){$\b x_1$} \put(5, 103){$\b x_4$}  \put(94,-11){$\b x_2$}  \put(84,103){$\b x_3$} 
    \put(47,33){$\b x_0$} 
 \end{picture} }

 \put(120,0){\begin{picture}(100,108)(0,0)  \put(0,0){\mq} 
    \put(100,50){\c}   \put(50,100){\c}   \put(0,50){\c}   \put(50,0){\c}
    \multiput(100,30)(0,40){2}{\h}  \multiput(0,30)(0,40){2}{\h} \multiput(30,100)(40,0){2}{\v} \multiput(30,0)(40,0){2}{\v} 
    \multiput(95,52)(2,0){2}{\h}   \multiput(-5,52)(2,0){2}{\h} 
    \multiput(52,-5)(0,2){2}{\v}  \multiput(52,95)(0,2){2}{\v}
 \end{picture} }

 \put(240,0){\begin{picture}(100,108)(0,0)  \put(0,0){\mq} 
    \multiput(100,40)(0,20){2}{\c}  \multiput(40,100)(20,00){2}{\c}   \multiput(0,40)(0,20){2}{\c}   \multiput(40,0)(20,0){2}{\c}
    \multiput(100,18)(0,30){3}{\h}  \multiput(18,100)(30,00){3}{\v}   \multiput(0,18)(0,30){3}{\h}   \multiput(18,0)(30,0){3}{\v} 
    \multiput(92,52)(2,0){3}{\h}   \multiput(-8,52)(2,0){3}{\h}  \multiput(52,-8)(0,2){3}{\v}    \multiput(52,92)(0,2){3}{\v}  
    \multiput(95,30)(2,0){2}{\h}  \multiput(95,70)(2,0){2}{\h}  \multiput(-5,30)(2,0){2}{\h}  \multiput(-5,70)(2,0){2}{\h} 
  \multiput(30,-5)(0,2){2}{\v}  \multiput(70,-5)(0,2){2}{\v}  \multiput(30,95)(0,2){2}{\v}  \multiput(70,95)(0,2){2}{\v} 
 \end{picture} }

 \end{picture}
 \caption{The 16, 28 and 44 degrees of freedom 
     of the $C^1$-$P_3$, $P_4$ and $P_5$ Fraeijs deVeubeke and Sander 
     macro-elements.} \label{f-dof}
 \end{figure}

The $C^1$-$P_3$ Fraeijs de Veubeke-Sander finite element is a
  macro-triangle element on general quadrilateral meshes,
    with 16 degrees of freedom per element, cf. \cite{Fraeijs,Sander}. 
On each quadrilateral, four $P_k$ polynomials are defined on
  the four triangles subdivided from the quadrilateral by 
  its two diagonal lines, see the left diagram in Figure \ref{f-dof}.
When proving the uni-solvency of Fraeijs de Veubeke-Sander finite element,
  three pieces of $P_3$ polynomials are employed in 
  the definition, cf. \cite{Ciarlet,Ciavaldini}. 
But it is not possible to construct $C^1$-$P_k$ ($k>3$)
  finite elements by three pieces of $P_k$ polynomials on one quadrilateral.
We uses four pieces of $P_k$ ($k\ge 3$) polynomials to construct the
  $C^1$ finite element, cf. Figure \ref{f-dof} for a display of degrees of freedom.
The first family of newly-constructed $C^1$-$P_k$ finite elements  is the full $C^1$-$P_k$
  space on the macro-mesh. 
Thus the element can be applied to interface problems.
The second family of $C^1$-$P_k$ finite elements condenses  
  all internal degrees of freedom by moving them to the four edges.
Thus the second element method has much less unknowns but is
  more accurate than the first one.
We prove the uni-solvency and the optimal order convergence.

Numerical tests on the two families of new $C^1$-$P_k$ elements
   and comparisons with the $C^1$-$P_k$ Argyris are provided,
   for solving the biharmonic equation and jump-coefficient biharmonic equations. 
The second family elements are more accurate with much less unknowns,
  comparing to the Argyris elements and the family elements, when solving the
   biharmonic equation.
But with coefficient jumps, the second family elements and the Argyris elements 
  deteriorate their convergence order to $0.5$ while
  the first family elements retain the optimal order $O(h^{k+1})$ of convergence.

\section{The  $C^1$-$P_k$   Fraeijs deVeubeke and Sanderfinite elements} 

We study the $C^1$-$P_k$ finite element space on a quadrilateral subdivided into four triangles.
  
We cut a general quadrilateral $\b x_1\b x_2\b x_3\b x_4$ \
   by its two diagonals into four triangles $T_1=\b x_1\b x_2\b x_0$, \
   $T_2=\b x_2\b x_3\b x_0$, \ $T_3=\b x_3\b x_4\b x_0$ and $T_4=\b x_4\b x_1\b x_0$,\
    cf. Figure \ref{f-dof}.
We denote the restriction of a function on the four triangles as 
\a{ p_i = p|_{T_i}, \quad i=1,2,3,4. }
We define the degrees of freedom of the  $C^1$-$P_k$
 \   Fraeijs deVeubeke and Sanderfinite elements, $k\ge 3$, \  by,
  for $i=1,2,3,4$,
\an{\label{dof} F_m(p) &=\begin{cases} p_i(\b x_i), \partial_x p_i(\b x_i), \partial_y p_i(\b x_i),  & \\
    \partial_{\b n} p_i(\frac { j\b x_i+(k-1-j)\b x_{i'}}{k-1}), & j=1,\dots, k-2, \\ 
     p_i(\frac j{k-2}\b x_i+\frac {k-2-j}{k-2}\b x_{i'}), & j=1,\dots, k-3, \\
    \partial_{\b n^2} p_i(\frac j{k-2}\b x_i+\frac {k-2-j}{k-2}\b x_{i'}), & j=1,\dots, k-3, \\ 
    \partial_{\b n^l} p_i(\frac j{k-l}\b x_i+\frac {k-l-j}{k-l}\b x_{i'}), & j=1,\dots, k-1-l,\\ & 
       \quad    l=3,\dots,k-2, 
    \end{cases} }
where $i'=i+1$, $\b x_5=\b x_1$, and $\b n$ is one fixed normal vector on the edge.  
    We note that the 4th line is  the special case of the 5th line when $l=2$.  
When $k=4$, the 5th line disappears.
If $k=3$, both the 4th and 5th lines disappear.
The total number of degrees of freedom in \eqref{dof} is
\an{ \label{t-dof}  \ad{ N&= 4 ( 3 + (k-2) + (k-3)+\sum_{l=2}^{k-2} (k-1-l) ) \\
                   &= 4( 2k-2 + \frac{(k-3)(k-2)}2 )=2k^2-2k+4.     }  }
Thus, when $k=3,4,5$, $N=16, 28, 44$, respectively,  as depicted in Figure \ref{f-dof}.

In order to get a $C^1$ element on the quadrilateral,  we post the following continuity constraints,
\an{\label{c} \begin{cases}  (p_i-p_{i'})(\b x_i)=0, & i=1,2,3,4, \\
           \partial_x (p_i-p_{i'})(\b x_i)=0,  & i=1,2,3,4, \\
           \partial_y (p_i-p_{i'})(\b x_i)=0,  & i=1,2,3,4, \\  
    (p_i-p_{i'})(\frac { j\b x_i+(k-2-j)\b x_{i'}}{k-2})=0, & i=1,2,3,4,  \\ &  j=1,\dots, k-3, \\  
    \partial_{\b x_0\b x_i^\perp} (p_i-p_{i'})(\frac { j\b x_i+(k-1-j)\b x_{i'}}{k-1})=0, & i=1,2,3,4, \\ & j=1,\dots, j_i, \\  
           p_1(\b x_0)=p_{i}(\b x_0), & i=2,3,4, \\
           \partial_{\b x_2\b x_4^\perp} (p_1- p_{i})(\b x_0)=0,  & i=2,3,4, \\
           \partial_{\b x_1\b x_3^\perp} ( p_1 - p_{i})(\b x_0)=0,  & i=2,3,4, \\ 
 \end{cases} }  where $i'=i+1$, $p_5=p_1$,  
   $\partial_{\b x_0\b x_i^\perp}$ is the directional derivative normal to edge $\b x_0\b x_i$, and
\an{\label{omit}
     j_1=j_2=j_4=k-2, \ j_3=k-3. }

We note that by \eqref{omit} we skip a constraint 
   $\partial_{\b x_0\b x_i^\perp} (p_i-p_{i'})(\frac { j\b x_i+(k-1-j)\b x_{i'}}{k-1})=0$ in \eqref{c}.
Thus, \eqref{c} does not ensure $p\in C^1( \b x_1\b x_2\b x_3\b x_4)$.
We will show that the missing constraint is automatically satisfied.

There are  \a{ N_e &= 4\times 3+4(2k-5-1 + 3\times 3=8k }
equations in \eqref{c}.  Combined with the dofs in \eqref{dof}, by \eqref{t-dof},  we get
\an{\label{N-e} N+N_e &=2k^2+6k+4=4\frac{(k+1)(k+2)}2 = 4\dim P_k. }
Thus the number of equations in \eqref{dof} and \eqref{c} is equal to the number of unknown coefficients.
We will prove in next section that there is a unique solution for each nodal basis function.

Let $\mathcal Q_h$ be a quasiuniform quadrilateral mesh of size $h$ on the polygonal domain $\Omega$.
On one quadrilateral $Q$,  we denote the basis functions by
\an{ \label{b-e}  \begin{cases} \phi_{i,Q}(\b x),  & i=1,\dots, 8k-8, \\
              \psi_{j,Q}(\b x), & j=1,\dots, (k-3)(k-2)/2, \end{cases} }
where $\{\phi_i\}$ are corresponding to the degrees of freedom of function values and first derivatives in the first three rows of
  \eqref{c}, and $\{\psi_j\}$ are   corresponding to the degrees of freedom of second and higher derivatives in the
  last two rows of \eqref{c}.  

Next, we define two types of global $C^1$-$P_k$  Fraeijs deVeubeke and Sanderfinite finite element spaces.
The first one is the full $C^1$-$P_k$  space on the macro-mesh:
\a{ V_h^{(1)} = \Big\{ \sum_{i=1}^{N_l} c_i \phi_i + \sum_{Q\in \mathcal Q_h} \sum_{j=1}^{(k-3)(k-2)/2} c_j \phi_{j,Q}
   \Big\},  }
where $\phi_i$ is a global basis function whose restriction on the neighboring quadrilateral is $\phi_{i,Q}$,
  and $N_l$ is the number of such non-local global basis functions. 
Restricting the boundary values and the boundary first derivatives to zero,  we define
\an{\label{V1} V_{h,0}^{(1)} = V_h^{(1)} \cap H^2_0(\Omega).  }

In this family of finite elements, we can replace the $(k-3)(k-2)/2$ high-order normal-derivative
  basis functions, $\phi_{j,Q}$ in \eqref{b-e}, at the four edges by the basis functions with $P_{k-4}$ internal Lagrange
   nodal degrees of freedom inside the triangle at the edge, i.e., 
    changing the dofs in the last two lines of \eqref{dof}.

The second finite element
   condenses the second and higher derivatives degrees of freedom
   so that the same accuracy is reached with much less unknowns globally.
In fact, with much less numbers of unknowns, the second method produces 
   much more accurate solutions than
  the first one, cf. the numerical tests at the end. 
\a{ V_h^{(2)} = \Big\{ \sum_{i=1}^{N_l} c_i \phi_i +  \sum_{i=1}^{N_h} c_j \phi_j 
   \Big\},  }
where $\phi_j$ is a global basis function whose restriction on the two neighboring quadrilateral is $\phi_{j,Q}$,
  and $N_h$ is the number of $2+$-derivative basis functions globally.
As some high derivatives are continuous across edges, thus,  $ V_h^{(2)}  \subsetneq  C^1$-$P_k$.
Restricting the boundary values and the boundary first derivatives to zero,  we define
\an{\label{V2}  V_{h,0}^{(2)} = V_h^{(2)} \cap H^2_0(\Omega).  }

We solve a model biharmonic equation on a bounded polygonal domain $\Omega$ in 2D:
Find $u$ such that
\an{\label{bi} \ad{ \Delta^2 u &= f \qquad \t{in } \ \Omega, \\
                           u=\partial_{\b n} u &=0  \qquad \t{on } \ \partial \Omega.  } }
            The finite element discretization on mesh $\mathcal Q_h$ is:
Find $u_h\in V_{h,0}^{(l)}$, $l=1$ or 2, such that
\an{\label{finite} ( \Delta  u, \Delta v) &= (f,v) \qquad \forall v\in V_{h,0}^{(l)},}
where $V_{h,0}^{(1)}$ and $V_{h,0}^{(2)}$ are defined in \eqref{V1} and \eqref{V2}, respectively.

\section{Uni-solvency and convergence}

We first present a non-trivial lemma, which will be used repeatedly
  for $P_3$, $P_4$, and $P_k$ polynomials.
 
\begin{lemma} 
Let a triangle $\b x_1\b x_2\b x_4$ be subdivided into 
  two triangles $T_1=\b x_1\b x_2\b x_0$ and $T_4=\b x_1\b x_0\b x_4$, cf. Figure \ref{1-e0}.
Let the $C^1$-$P_k$ function $p$ satisfy, for any $0<m\le k-2$,
\an{\label{1-k} \begin{cases} p_1=p|_{T_1}, \ \partial_{(\b x_1\b x_2^\perp)^i} p_1|_{\b x_1\b x_2}=0, \quad \ i =0,1,\dots,m, \\
          p_4=p|_{T_4}, \ \partial_{(\b x_1\b x_4^\perp)^i} p_4|_{\b x_1\b x_4}=0, \quad \  i =0,1,\dots,m,\\
          \partial_{(\b x_1\b x_0^\perp)^i} p_1|_{\b x_1\b x_0}
       =\partial_{(\b x_1\b x_0^\perp)^i} p_4|_{\b x_1\b x_0}, \ i=0,1,  \end{cases}
  } where $ \partial_{(\b x_1\b x_2^\perp)^i} p_1$ is the $i$-th normal derivative to the edge. 
  Then, the $m+1$ tangential derivative at $\b x_1$ vanishes, i.e.,
  \an{\label{1-k2}  \partial_{(\b x_1\b x_0)^{m+1}} p_1(\b x_0) = 0. }
\end{lemma}

\begin{proof} As the function value and all directional derivatives are invariant under shifting and rotating,
   we assume the triangle is located at the origin, shown as in Figure \ref{1-e0}.

\begin{figure}[H]\centering \setlength\unitlength{1.3pt} 
 \begin{picture}(150,100)(-20,0) 
 
 \put(0,0){\begin{picture}(100,108)(0,0) 

 \put(0,30){\line(1,0){50}}  \put(0,30){\line(1,2){30}}  \put(0,30){\line(2,-1){60}} 
 \put(60,0){\line(-1,3){30}}  
   \put(50,33){$\b x_0$}    \put(-22, 20){$(0,0)$}
   \put(-17, 30){$\b x_1$} \put(61,6){$\b x_2$}  \put(35,90){$\b x_4$}   \put(25,40){$T_4$}  \put(39,15){$T_1$} 
 \end{picture} } 
 \end{picture}
 \caption{A triangle $\b x_1\b x_2\b x_4$ at the origin $\b x_1(0,0)$ is split into two by the   horizontal
    line $\b x_1\b x_0$.} \label{1-e0}
 \end{figure}
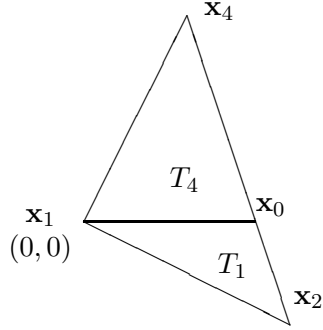
 
   By \eqref{1-k},  as the function $p_1$ and its first $m$ 
       normal derivatives vanish on the line $\b x_1\b x_2$, we have  
\an{\label{p14} \ad{ p_1 &= (y+a_1 x)^{m+1} \sum_{i+j\le k-m-1} c_{ij} x^i y^j 
      \ \t{ \ for some \ } a_1>0, c_{ij}, \\
               p_4&= (y-a_2 x)^{m+1} \sum_{i+j\le k-m-1} d_{ij} x^i y^j  \ \t{ \ for some \ } a_2>0, d_{ij} .  } }
By the $C^1$ condition in \eqref{1-k}, we get, by comparing coefficients, from \eqref{p14}, that
\a{ a_1^{m+1} c_{i0} &= (-a_2)^{m+1} d_{i0},  \quad  i=0,\dots,k-m-1, \\
     (m+1) a_1^m c_{00}   &=(m+1)(- a_2)^m c_{00} ,   \\
     (m+1) a_1^m c_{i+1,0} + a_1^{m+1} c_{i1} &=(m+1)(- a_2)^m d_{i+1,0} +   (-a_2)^{m+1} d_{i1} }
for $i=0,\dots, k-m-2$.  By the first two equations above, we get  
   \an{\label{3-0} \left\{ \ad{ a_1^{m+1} c_{00} + (-1)^{m} a_2^{m+1} d_{00} & =0, \\
                  a_1^m \ \ \; c_{00} - (-1)^{m} a_2^m \ \ \; d_{00} & = 0 , } \right. \quad  \Rightarrow \
                   \ \left\{ \ad{   c_{00} & =0, \\
                 d_{00} & = 0 , } \right. }
   as the determinant is $(-1)^{m+1}a_1^m a_2^m (a_1+a_2)$,  non-zero.  
This leads to an additional constraint that the $m+1$st tangential derivative 
\a{  &\quad \ \partial_{(\b x_1\b x_0)^{m+1}} p_1(\b x_0) =  \partial_{y^{m+1}} p_1(0,0) \\
         &= (m+1)! c_{00} + 0 \sum_{j=1}^{m+1} \p{ m+1 \\ j } 
         \partial_{x^j} \left(\sum_{i+j\le k-m-1} c_{ij} x^i y^j \right)\\
         &=0.  }
The lemma is proved.
\end{proof}
  
\begin{lemma}  
Let a quadrilateral $Q= \b x_1\b x_2\b x_3\b x_4$ be subdivided into 
  four triangles by its two diagonals, where $\b x_0$ is the intersection of $\b x_1\b x_3$
    and $\b x_2\b x_4$, cf. Figure \ref{1-e4}.
The $(2k^2-2k+4)$ degrees of freedom in \eqref{dof} and $8k$ constraints in \eqref{c} 
  uniquely define a $C^1$-$P_k$ function on $Q$. 
\end{lemma}

\begin{proof} By \eqref{N-e}, we have a square system of linear equations.
The uni-solvency is implied by uniqueness.  
For the $p$ satisfying the above $4\dim P_k$  homogeneous equations,  we show $p=0$.

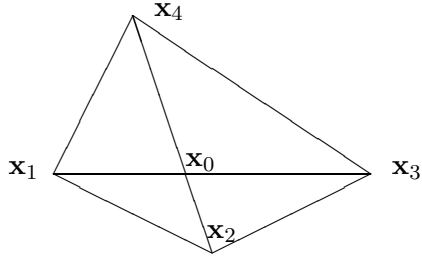
\begin{figure}[H]\centering
 \begin{picture}(150,100)(-20,0) 
 
 \put(0,0){\begin{picture}(100,108)(0,0) 

 \put(0,30){\line(1,0){120}}  \put(0,30){\line(1,2){30}}  \put(0,30){\line(2,-1){60}} 
 \put(60,0){\line(-1,3){30}}   \put(60,0){\line(2,1){60}}   \put(120,30){\line(-3,2){90}}  
   \put(50,33){$\b x_0$}  
   \put(-17, 30){$\b x_1$} \put(58,6){$\b x_2$}  \put(128,30){$\b x_3$}  \put(38,90){$\b x_4$} 
 \end{picture} } 
 \end{picture}
 \caption{A quadrilateral $Q=\b x_1\b x_2\b x_3\b x_4$ is subdivided into 
    four triangles by its two diagonals.} \label{1-e4}
 \end{figure}

On edge $\b x_1\b x_2$, the $P_k$ polynomial $p$ has its values
   and its tangential derivatives at two end-points all zero, and have $(k-3)$ mid-edge values zero.  We have $p|_{\b x_1\b x_2}=0$.
Again, on $\b x_1\b x_2$, the $P_{k-1}$ 
   polynomial, the normal derivative,  $\partial_{\b x_1\b x_2^\perp}p|_{\b x_1\b x_2}$ vanishes 
   due to the  two end-points and the mid-edge normal derivative conditions in \eqref{dof}.
Thus \an{\label{20}  p|_{\b x_1\b x_2} = \partial_{{\b x_1\b x_2^\perp}}p|_{\b x_1\b x_2}=0, }
where $\partial_{\b x_1\b x_2^\perp}$ denotes a directional derivative orthogonal to $\b x_1\b x_2$.
 By \eqref{20}, we have
\an{\label{l20} p|_{\b x_1\b x_2\b x_0} = \lambda_1^2 p_{k-2} 
   \quad \t{ for some \ } p_{k-2}\in P_{k-2}(\b x_1\b x_2\b x_0),
   }  where $\lambda_1$ is a linear function vanishing on $\b x_1\b x_2$ and assuming value 1 at $\b x_0$.

By \eqref{1-k2}, we have a second order vanishing tangential derivative at the corner, 
  cf. Figure \ref{1-e4},
\an{\label{p1d}  \partial_{\b x_1\b x_0^2} (p|_{\b x_1\b x_2\b x_0})(\b x_1) =0. }
As the normal derivative is identically 0 on the edge $\b x_1\b x_2$, 
         $\partial_{\b x_1\b x_2^\perp}p|_{\b x_1\b x_2}=0$,
          its tangent derivative vanishes on the edge,  $\partial_{(\b x_1\b x_2^\perp)
  (\b x_1\b x_2)}p|_{\b x_1\b x_2}=0$.
In particular,
\an{\label{p2d}  \partial_{(\b x_1\b x_2^\perp)(\b x_1\b x_2)} (p|_{\b x_1\b x_2\b x_0})(\b x_1) =0. }

Let the unit vector, for some unit vector $\langle v_1, v_2\rangle$ with $v_2\ne 0$, cf. Figure \ref{1-e4},
\a{ \frac{\b x_1\b x_0}{\|\b x_1\b x_0\|} = v_1  \frac{\b x_1\b x_2}{\|\b x_1\b x_2\|}  
     + v_2 \frac{\b x_1\b x_2^\perp}{\|\b x_1\b x_2^\perp\|} .  }
Thus, by \eqref{p1d} and \eqref{p2d}, we have
\a{ 0 & =  \partial_{\b x_1\b x_0^2} (p|_{\b x_1\b x_2\b x_0})(\b x_1) \\
        &= v_1^2  \partial_{\b x_1\b x_2^2} (p|_{\b x_1\b x_2\b x_0})(\b x_1) \\
         & \quad \ + 2v_1v_2
                  \partial_{(\b x_1\b x_2^\perp)(\b x_1\b x_2)} (p|_{\b x_1\b x_2\b x_0})(\b x_1) \\
         & \quad \         + v_2^2
                  \partial_{(\b x_1\b x_2^\perp)^2} (p|_{\b x_1\b x_2\b x_0})(\b x_1) \\
        &=  v_2^2  \partial_{(\b x_1\b x_2^\perp)^2} (p|_{\b x_1\b x_2\b x_0})(\b x_1),          }
which implies 
\an{ \label{l0} \partial_{(\b x_1\b x_2^\perp)^2} (p|_{\b x_1\b x_2\b x_0})(\b x_1)=0. }

At edge $\b x_2\b x_0$, the function $\partial_{\b x_2\b x_0^\perp} p$ is continuous by \eqref{c}.
Repeating above work for \eqref{l0} on edge $\b x_2\b x_0$, we get
\an{\label{r0} \partial_{(\b x_1\b x_2^\perp)^2} (p|_{\b x_1\b x_2\b x_0})(\b x_2)=0. }
By \eqref{l0}, \eqref{r0} and the degrees of freedom in the 4th line of \eqref{dof},
  as the $P_{k-2}$ polynomial $\partial_{(\b x_1\b x_2^\perp)^2}(p|_{\b x_1\b x_2\b x_0})$ has
    $k-1$ zero points on the line,  
  we have
  \a{ \partial_{(\b x_1\b x_2^\perp)^2} p|_{\b x_1\b x_2 }=0.  }
With it, we rewrite \eqref{l20} as 
\an{\label{l30}  p|_{\b x_1\b x_2\b x_0} = \lambda_1^3 p_{k-3} 
   \quad \t{ for some \ } p_{k-3}\in P_{k-3}(\b x_1\b x_2\b x_0).
   } 
                  
 If $k>3$, the above $p_{k-3}$ is not a constant yet.
We have $m=3$ now in \eqref{1-k} by \eqref{l30}.
We repeat all the work above on $ \partial_{(\b x_1\b x_2^\perp)^3} p|_{\b x_1\b x_2 }$ to obtain
  \a{ \ad{ & \partial_{(\b x_1\b x_2^\perp)^3} p|_{\b x_1\b x_2 }=0, \ \t{and} \\
         & p|_{\b x_1\b x_2\b x_0} = \lambda_1^4 p_{k-4} 
   \quad \t{ for some \ } p_{k-4}\in P_{k-4}(\b x_1\b x_2\b x_0).
   } }
It would be repeated until
\an{\label{l01}   p|_{\b x_1\b x_2\b x_0} = \lambda_1^{k} c_1 
   \quad \t{ for some \ } c_1\in P_0(\b x_1\b x_2\b x_0).  }

On triangle $\b x_4\b x_1\b x_0$, cf. Figure \ref{1-e4},  as $\partial_{\b x_4\b x_0^\perp} p$ is continuous 
  across edge $\b x_0\b x_4$,  
   we have 
\an{\label{l04}   p|_{\b x_4\b x_1\b x_0} = \lambda_4^{k} c_4 
   \quad \t{ for some \ } c_4\in P_0(\b x_4\b x_1\b x_0),  }
 where $\lambda_4$ is a linear function vanishing on $\b x_4\b x_1$ and assuming value 1
   at $\b x_0$.

By \eqref{l01} and \eqref{l04}, as $p$ is continuous, 
   we get
\an{\label{c-e} c_1=c_4.  } 
By \eqref{l01}, \eqref{l04} and \eqref{c-e}, as $\partial_{\b x_2\b x_4}p$ is continuous on two triangles $\b x_4\b x_1\b x_0$
   and $\b x_1\b x_2\b x_0$ (not known yet on the other two triangles), 
   we get
\an{\label{c-e1} \ad{ k \lambda_1^{k-1} c_1 \|\b x_1\b x_0\|^{-1}&=
           - k \lambda_1^{k-1} c_4 \|\b x_4\b x_0\|^{-1}, \\
                 c_1 &= - c_1 \frac{\|\b x_1\b x_0\|}{\|\b x_4\b x_0\|} = 0.  } }

Thus $p=0$ on the two triangles, $\b x_1\b x_2\b x_0$ and $\b x_4\b x_1\b x_0$.
We could not use this proof on the other two triangles due to a missing
   normal derivative constraint in \eqref{omit}.

By the $C^1$ continuity across $\b x_2\b x_0$, as $p=0$ on the other side of the edge,  we get
\a{ \partial_{\b x_2\b x_4^{\perp}} (\partial_{\b x_0\b x_3} p_2 )(\b x_0) = 0.  } 
By the $C^1$ continuity across $\b x_4\b x_0$, we get
\a{ \partial_{\b x_2\b x_4^{\perp}} (\partial_{\b x_0\b x_3} p_3 )(\b x_0) = 0.  }
Together, they imply that the two above $P_{k-1}$ polynomials,
  $\partial_{\b x_0\b x_3} p_2$ and $\partial_{\b x_0\b x_3} p_3$, have both tangent derivatives  
   vanishing at $\b x_0$, at the two sides of edge $\b x_0\b x_3$.
By \eqref{c}, in addition, these two polynomials are continuous at $k-1$ points on the edge
    $\b x_0\b x_3$.
Thus they are continuous at $\b x_0\b x_3$, i.e., $p$ is also $C^1$ across $\b x_3\b x_0$.
Repeating the work on the other two triangles, we get $p =0$.  The lemma is proved. 
 \end{proof}

\begin{theorem}  Let $ u\in H^{k+1}$ be the exact solution of the biharmoic equation \eqref{bi}.  
   Let $u_h$ be the $C^1$-$P_k$ finite element solution of \eqref{finite}.   
   Assuming the full-regularity on \eqref{bi}, it holds 
  \a{  \| u- u_h\|_{0} + h  |  u- u_h |_{1}+ h^2  |  u- u_h |_{2}  
         & \le Ch^{k+1} | u|_{k+1}, \quad k\ge3.  } 
\end{theorem}
                        
\begin{proof} The proof is standard.  One proof can be found in \cite{Zhang-bubble-p4}.
\end{proof}

\section{Numerical Experiments}

In the numerical computation,  we solve the following interface problem on domain $\Omega=(0,1)\times(0,1)$. 
Find $u\in H^2_0(\Omega)$, such that 
\an{ \label{s1} \t{div}\b{div}( \mu D^2 u ) & = f, \t{ \ where }\ \mu(x,y)=\begin{cases} \mu_0, &\t{ if } \ x\le \frac 12,
    \\ 1, &\t{ if } \ x > \frac 12. \end{cases}
  }
When $\mu_0=1$, \eqref{s1} is the standard biharmonic equation \eqref{bi}.
We choose an $f$ in \eqref{s1} so that the exact solution is
\an{\label{s2}
   u =\begin{cases} - \mu_0 y^4 (y-1)^4 x^2  (4 x-3) (2 x-1 )^2, &\t{ if } \ x\le \frac 12,
    \\(x-1 )^2 y^4 (y-1)^4 (4 x-1) (2 x-1)^2, &\t{ if } \ x > \frac 12. \end{cases} } 
Here $u$ is not smooth at the vertical line $x=1/2$ if $\mu_0\ne 1$.  
That is, the second and high orders derivatives jump across the line.

     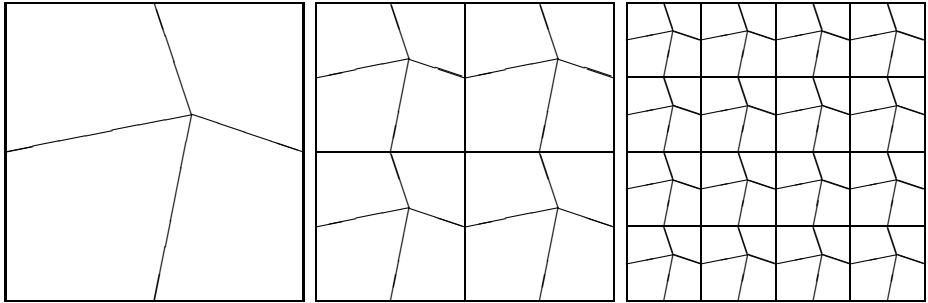
\begin{figure}[H] \setlength\unitlength{0.94pt}\begin{center}
    \begin{picture}(370,120)(0,0)
     \def\mc{\begin{picture}(120,120)(0,0)
       \put(0,0){\line(1,0){120}} \put(0,120){\line(1,0){120}}
      \put(0,0){\line(0,1){120}}  \put(120,0){\line(0,1){120}}  
      \put(75,75){\line(-1,-5){15}}  \put(75,75){\line(-5,-1){75}}
     \put(75,75){\line(-1,3){15}}  \put(75,75){\line(3,-1){45}}
      \end{picture}}

    \put(0,0){\mc}
    
      \put(125,0){\setlength\unitlength{0.47pt}\begin{picture}(90,90)(0,0)
    \put(0,0){\mc}\put(120,0){\mc}\put(0,120){\mc}\put(120,120){\mc}\end{picture}}
      \put(250,0){\setlength\unitlength{0.235pt}\begin{picture}(120,120)(0,0)
    \multiput(0,0)(120,0){4}{\multiput(0,0)(0,120){4}{\mc}} \end{picture}}
    \end{picture}
 \caption{The first three grids for the computation in Tables \ref{t-1}, \ref{t-3}, \ref{t-4}, \ref{t-6} and 
         \ref{t-7}. }\label{q-grid} 
    \end{center} \end{figure}

We compute the solution \eqref{s2} on the quadrilateral grids shown in Figure \ref{q-grid}, by 
  the newly constructed $C^1$-$P_k$ Fraeijs de Veubeke-Sander finite elements.
The results, by the  $C^1$-$P_3$ finite element, 
   are listed in Table \ref{t-1}, where we use the notation of weighted norm,
  \a{ | u |_a ^2 =(\mu D^2 u, D^2 u). }
We can see that the optimal orders of convergence 
  are achieved in all cases, independent of the coefficient jump.

\begin{table}[H]
  \centering  \renewcommand{\arraystretch}{1.1}
  \caption{Error profile by the $C^1$-$P_3$ element on the  
   quadrilateral meshes shown as in Figure \ref{q-grid}, for computing \eqref{s2}. }
  \label{t-1}
\begin{tabular}{c|cc|cc|cc}
\hline
grid & \multicolumn{2}{c|}{ $\| u-u_h\|_{0}$  \; $O(h^r)$} 
   & \multicolumn{2}{c|}{  $| u-u_h|_{1}$ \;$O(h^r)$} 
  &  \multicolumn{2}{c}{  $|u-u_h|_{a}$ \;$O(h^r)$} \\ \hline
    &  \multicolumn{6}{c}{ By the quadrilateral $C^1$-$P_3$ \eqref{V1} element, $\mu_0=1$. }   \\
\hline   
 3&     0.466E-6 &  3.7&     0.487E-5 &  3.1&     0.219E-3 &  1.8 \\
 4&     0.312E-7 &  3.9&     0.583E-6 &  3.1&     0.563E-4 &  2.0 \\
 5&     0.197E-8 &  4.0&     0.754E-7 &  2.9&     0.147E-4 &  1.9 \\
 6&     0.123E-9 &  4.0&     0.976E-8 &  2.9&     0.380E-5 &  2.0 \\
\hline 
 &  \multicolumn{6}{c}{ By the quadrilateral $C^1$-$P_3$ \eqref{V1} element, $\mu_0=10$. }   \\
\hline  
 3&     0.466E-6 &  3.7&     0.487E-5 &  3.1&     0.219E-3 &  1.8 \\
 4&     0.312E-7 &  3.9&     0.583E-6 &  3.1&     0.563E-4 &  2.0 \\
 5&     0.197E-8 &  4.0&     0.754E-7 &  2.9&     0.147E-4 &  1.9 \\
 6&     0.123E-9 &  4.0&     0.976E-8 &  2.9&     0.380E-5 &  2.0 \\
\hline   
    \end{tabular}%
\end{table}%

As a comparison, we compute the solution \eqref{s2} on uniform triangular grids shown 
   in Figure \ref{t-grid}, by the $C^1$-$P_4$ Bell finite element.
The results,  are listed in Table \ref{t-2}.
We can see that the optimal orders of convergence 
  are achieved in all norms when $\mu_0$.
But for the interface problem, the order of convergence is reduced to a half, as the solution is
  in $H^{2+1/2}(\Omega)$.  
Comparing to Table \ref{t-1}, we can see an advantage of the Fraeijs de Veubeke-Sander finite elements.

     \begin{figure}[H] \setlength\unitlength{1.1pt}\begin{center}
    \begin{picture}(300,90)(0,0)
     \def\mc{\begin{picture}(90,90)(0,0)
       \put(0,0){\line(1,0){90}} \put(0,90){\line(1,0){90}}
      \put(0,0){\line(0,1){90}}  \put(90,0){\line(0,1){90}} 
     \put(0,90){\line(1,-1){90}}
      \end{picture}}

    \put(0,0){\mc}
      \put(105,0){\setlength\unitlength{0.55pt}\begin{picture}(90,90)(0,0)
    \put(0,0){\mc}\put(90,0){\mc}\put(0,90){\mc}\put(90,90){\mc}\end{picture}}
      \put(210,0){\setlength\unitlength{0.275pt}\begin{picture}(90,90)(0,0)
    \multiput(0,0)(90,0){4}{\multiput(0,0)(0,90){4}{\mc}} \end{picture}}
    \end{picture}
 \caption{The first three levels of grids for the computation in Tables \ref{t-2}, \ref{t-5} and \ref{t-8}. }\label{t-grid} 
    \end{center} \end{figure}
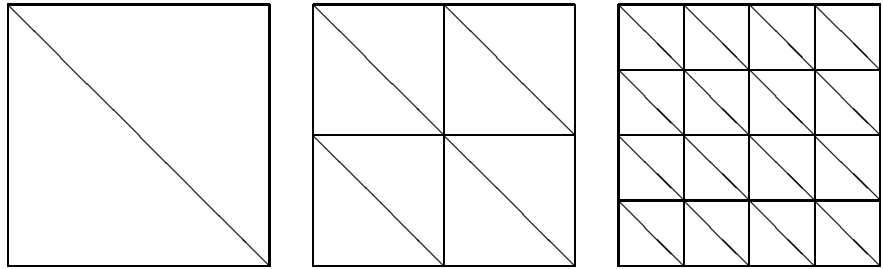

\begin{table}[H]
  \centering  \renewcommand{\arraystretch}{1.1}
  \caption{Error profile by the $C^1$-$P_4$ Bell element on the  
   triangular meshes shown as in Figure \ref{t-grid}, for computing \eqref{s2}. }
  \label{t-2}
\begin{tabular}{c|cc|cc|cc}
\hline
grid & \multicolumn{2}{c|}{ $\| u-u_h\|_{0}$  \; $O(h^r)$} 
   & \multicolumn{2}{c|}{  $| u-u_h|_{1}$ \;$O(h^r)$} 
  &  \multicolumn{2}{c}{  $|u-u_h|_{a}$ \;$O(h^r)$} \\ \hline
    &  \multicolumn{6}{c}{ By the triangular $C^1$-$P_4$ Bell element, $\mu_0=1$. }   \\
\hline  
 2&   0.2825E-03 &  --- &   0.1413E-02 &  --- &   0.1147E-01 &  ---  \\
 3&   0.1029E-05 &  8.10&   0.2060E-04 &  6.10&   0.7314E-03 &  3.97 \\
 4&   0.2497E-07 &  5.37&   0.1139E-05 &  4.18&   0.9827E-04 &  2.90 \\
 5&   0.6027E-09 &  5.37&   0.6576E-07 &  4.11&   0.1251E-04 &  2.97 \\
\hline 
 &  \multicolumn{6}{c}{ By the triangular $C^1$-$P_4$ Bell element, $\mu_0=10$. }   \\
\hline 
 2&   0.1248E-02 &  --- &   0.7027E-02 &  --- &   0.8385E-01 &  ---  \\
 3&   0.1863E-04 &  6.07&   0.3254E-03 &  4.43&   0.2558E-01 &  1.71 \\
 4&   0.8777E-05 &  1.09&   0.1141E-03 &  1.51&   0.1632E-01 &  0.65 \\
 5&   0.4494E-05 &  0.97&   0.4792E-04 &  1.25&   0.1081E-01 &  0.59 \\
\hline   
    \end{tabular}%
\end{table}%

We compute the solution \eqref{s2} on the quadrilateral grids shown in Figure \ref{q-grid}, by 
  the newly constructed $C^1$-$P_4$ Fraeijs de Veubeke-Sander finite element \eqref{V1} 
  in Table \ref{t-3}.
We can see that the optimal orders of convergence 
  are achieved in all cases, independent of the coefficient jump.

\begin{table}[H]
  \centering  \renewcommand{\arraystretch}{1.1}
  \caption{Error profile by the $C^1$-$P_4$ element on the  
   quadrilateral meshes shown as in Figure \ref{q-grid}, for computing \eqref{s2}. }
  \label{t-3}
\begin{tabular}{c|cc|cc|cc}
\hline
grid & \multicolumn{2}{c|}{ $\| u-u_h\|_{0}$  \; $O(h^r)$} 
   & \multicolumn{2}{c|}{  $| u-u_h|_{1}$ \;$O(h^r)$} 
  &  \multicolumn{2}{c}{  $|u-u_h|_{a}$ \;$O(h^r)$} \\ \hline
    &  \multicolumn{6}{c}{ By the quadrilateral $C^1$-$P_4$ \eqref{V1} element, $\mu_0=1$. }   \\
\hline   
 2&     0.429E-5 &  ---&     0.850E-4 &  ---&     0.372E-2 &  --- \\
 3&     0.134E-6 &  5.0&     0.638E-5 &  3.7&     0.548E-3 &  2.8 \\
 4&     0.432E-8 &  5.0&     0.433E-6 &  3.9&     0.736E-4 &  2.9 \\
 5&     0.136E-9 &  5.0&     0.277E-7 &  4.0&     0.941E-5 &  3.0 \\
\hline 
 &  \multicolumn{6}{c}{ By the quadrilateral $C^1$-$P_4$ \eqref{V1} element, $\mu_0=10$. }   \\
\hline  
 2&     0.429E-5 &  ---&     0.850E-4 &  ---&     0.372E-2 &  --- \\
 3&     0.134E-6 &  5.0&     0.638E-5 &  3.7&     0.548E-3 &  2.8 \\
 4&     0.432E-8 &  5.0&     0.433E-6 &  3.9&     0.736E-4 &  2.9 \\
 5&     0.136E-9 &  5.0&     0.277E-7 &  4.0&     0.941E-5 &  3.0 \\
\hline   
    \end{tabular}%
\end{table}%

We compute the solution \eqref{s2} on the quadrilateral grids shown in Figure \ref{q-grid}, by 
  the newly constructed condensed $C^1$-$P_4$ Fraeijs de Veubeke-Sander finite element \eqref{V2} 
  in Table \ref{t-4}.
We can see, when $\mu_0=1$, the condensed method has much less errors than its counterpart,
  the $C^1$-$P_4$ Fraeijs de Veubeke-Sander finite elements \eqref{V1}, comparing the first part
   of Table \ref{t-3} and \ref{t-4}. 
Here we reach the computer accuracy on grid 4, one level less than we did by the 
   standard $C^1$-$P_4$ Fraeijs de Veubeke-Sander finite element \eqref{V1}.
But when the solution is piecewise smooth, in the
  second part of Table \ref{t-4}, the condensed $C^1$-$P_4$ method is useless.

\begin{table}[H]
  \centering  \renewcommand{\arraystretch}{1.1}
  \caption{Error profile by the $C^1$-$P_4$ element on the  
   quadrilateral meshes shown as in Figure \ref{q-grid}, for computing \eqref{s2}. }
  \label{t-4}
\begin{tabular}{c|cc|cc|cc}
\hline
grid & \multicolumn{2}{c|}{ $\| u-u_h\|_{0}$  \; $O(h^r)$} 
   & \multicolumn{2}{c|}{  $| u-u_h|_{1}$ \;$O(h^r)$} 
  &  \multicolumn{2}{c}{  $|u-u_h|_{a}$ \;$O(h^r)$} \\ \hline
    &  \multicolumn{6}{c}{ By the condensed $C^1$-$P_4$ \eqref{V2} element, $\mu_0=1$. }   \\
\hline   
 2&    0.460E-06 &  --- &    0.889E-05 &  --- &    0.392E-03 & ---  \\
 3&    0.139E-07 &  5.0&    0.709E-06 &  3.6&    0.634E-04 &  2.6 \\
 4&    0.500E-09 &  4.8&    0.512E-07 &  3.8&    0.886E-05 &  2.8 \\
\hline 
 &  \multicolumn{6}{c}{ By the condensed $C^1$-$P_4$ \eqref{V2} element, $\mu_0=10$. }   \\
\hline  
 2&    0.157E-04 &  --- &    0.223E-03 &  --- &    0.143E-01 &  ---  \\
 3&    0.635E-05 &  1.3&    0.824E-04 &  1.4&    0.122E-01 &  0.2 \\
 4&    0.313E-05 &  1.0&    0.350E-04 &  1.2&    0.970E-02 &  0.3 \\
 5&    0.156E-05 &  1.0&    0.156E-04 &  1.2&    0.714E-02 &  0.4 \\
\hline   
    \end{tabular}%
\end{table}%

In Table \ref{t-5}, we compute the solution \eqref{s2} on uniform triangular grids shown 
   in Figure \ref{t-grid}, by the $C^1$-$P_5$ Argyris finite element. 
We can see that the optimal orders of convergence 
  are achieved in all norms when $\mu_0=1$.
Again, for the interface problem, the order of convergence is reduced to a half.

\begin{table}[H]
  \centering  \renewcommand{\arraystretch}{1.1}
  \caption{Error profile by the $C^1$-$P_5$ Argyris element on the  
   triangular meshes shown as in Figure \ref{t-grid}, for computing \eqref{s2}. }
  \label{t-5}
\begin{tabular}{c|cc|cc|cc}
\hline
grid & \multicolumn{2}{c|}{ $\| u-u_h\|_{0}$  \; $O(h^r)$} 
   & \multicolumn{2}{c|}{  $| u-u_h|_{1}$ \;$O(h^r)$} 
  &  \multicolumn{2}{c}{  $|u-u_h|_{a}$ \;$O(h^r)$} \\ \hline
    &  \multicolumn{6}{c}{ By the triangular $C^1$-$P_5$ Argyris element, $\mu_0=1$. }   \\
\hline  
 2&    0.273E-03 &  --- &    0.130E-02 &  --- &    0.106E-01 &  ---  \\
 3&    0.793E-06 &  8.43&    0.156E-04 &  6.39&    0.499E-03 &  4.42  \\
 4&    0.157E-07 &  5.66&    0.667E-06 &  4.54&    0.407E-04 &  3.61  \\
 5&    0.283E-09 &  5.79&    0.240E-07 &  4.80&    0.292E-05 &  3.80  \\
\hline 
 &  \multicolumn{6}{c}{ By the triangular $C^1$-$P_5$ Argyris element, $\mu_0=10$. }   \\
\hline 
 2&    0.119E-02 &  0.00&    0.644E-02 &  0.00&    0.738E-01 &  0.00  \\
 3&    0.125E-04 &  6.58&    0.249E-03 &  4.69&    0.193E-01 &  1.94  \\
 4&    0.512E-05 &  1.28&    0.831E-04 &  1.58&    0.114E-01 &  0.76  \\
 5&    0.265E-05 &  0.95&    0.328E-04 &  1.34&    0.720E-02 &  0.66  \\
\hline   
    \end{tabular}%
\end{table}%

In Table \ref{t-6}, we compute the solution \eqref{s2} 
   on the quadrilateral grids shown in Figure \ref{q-grid}, by 
  the newly constructed $C^1$-$P_5$ Fraeijs de Veubeke-Sander finite element \eqref{V1}.
We can see that the optimal orders of convergence 
  are achieved in all cases, independent of the coefficient jump.

\begin{table}[H]
  \centering  \renewcommand{\arraystretch}{1.1}
  \caption{Error profile by the $C^1$-$P_5$ element on the  
   quadrilateral meshes shown as in Figure \ref{q-grid}, for computing \eqref{s2}. }
  \label{t-6}
\begin{tabular}{c|cc|cc|cc}
\hline
grid & \multicolumn{2}{c|}{ $\| u-u_h\|_{0}$  \; $O(h^r)$} 
   & \multicolumn{2}{c|}{  $| u-u_h|_{1}$ \;$O(h^r)$} 
  &  \multicolumn{2}{c}{  $|u-u_h|_{a}$ \;$O(h^r)$} \\ \hline
    &  \multicolumn{6}{c}{ By the quadrilateral $C^1$-$P_5$ \eqref{V1} element, $\mu_0=1$. }   \\
\hline   
 2&     0.312E-6 &  ---&     0.107E-4 &  ---&     0.595E-3 &  --- \\
 3&     0.611E-8 &  5.7&     0.414E-6 &  4.7&     0.460E-4 &  3.7 \\
 4&     0.100E-9 &  5.9&     0.136E-7 &  4.9&     0.303E-5 &  3.9 \\
\hline 
 &  \multicolumn{6}{c}{ By the quadrilateral $C^1$-$P_5$ \eqref{V1} element, $\mu_0=10$. }   \\
\hline  
 2&     0.312E-6 &  ---&     0.107E-4 &  ---&     0.595E-3 &  --- \\
 3&     0.611E-8 &  5.7&     0.414E-6 &  4.7&     0.460E-4 &  3.7 \\
 4&     0.100E-9 &  5.9&     0.136E-7 &  4.9&     0.303E-5 &  3.9 \\
\hline   
    \end{tabular}%
\end{table}%

In Table \ref{t-7}, we  compute the solution \eqref{s2} on the quadrilateral grids shown in Figure \ref{q-grid}, by 
  the condensed $C^1$-$P_5$ Fraeijs de Veubeke-Sander finite element \eqref{V2}.
Again, when $\mu_0=1$, the condensed method is much more accurate and reaches 
  the computer accuracy on two grids.
But when the solution is piecewise smooth, in the
  second part of Table \ref{t-7}, the condensed $C^1$-$P_5$ method is nearly useless.

\begin{table}[H]
  \centering  \renewcommand{\arraystretch}{1.1}
  \caption{Error profile by the $C^1$-$P_5$ element on the  
   quadrilateral meshes shown as in Figure \ref{q-grid}, for computing \eqref{s2}. }
  \label{t-7}
\begin{tabular}{c|cc|cc|cc}
\hline
grid & \multicolumn{2}{c|}{ $\| u-u_h\|_{0}$  \; $O(h^r)$} 
   & \multicolumn{2}{c|}{  $| u-u_h|_{1}$ \;$O(h^r)$} 
  &  \multicolumn{2}{c}{  $|u-u_h|_{a}$ \;$O(h^r)$} \\ \hline
    &  \multicolumn{6}{c}{ By the condensed $C^1$-$P_5$ \eqref{V2} element, $\mu_0=1$. }   \\
\hline   
 2&    0.648E-07 &  ---&    0.166E-05 &  ---&    0.101E-03 &  --- \\
 3&    0.997E-09 &  6.0&    0.625E-07 &  4.7&    0.665E-05 &  3.9 \\
\hline 
 &  \multicolumn{6}{c}{ By the condensed $C^1$-$P_5$ \eqref{V2} element, $\mu_0=10$. }   \\
\hline  
 2&    0.168E-04 &  ---&    0.320E-03 &  ---&    0.295E-01 &  --- \\
 3&    0.598E-05 &  1.5&    0.121E-03 &  1.4&    0.263E-01 &  0.2 \\
 4&    0.280E-05 &  1.1&    0.467E-04 &  1.4&    0.207E-01 &  0.3 \\
 5&    0.138E-05 &  1.0&    0.186E-04 &  1.3&    0.152E-01 &  0.4 \\\hline   
    \end{tabular}%
\end{table}%

Finally, in Table \ref{t-8}, we compute the solution \eqref{s2} on uniform triangular grids shown 
   in Figure \ref{t-grid}, by the $C^1$-$P_6$ Argyris finite element. 
We can see that the optimal orders of convergence 
  are achieved in all norms when $\mu_0=1$.
Again, for the interface problem, the order of convergence is reduced to a half.

\begin{table}[H]
  \centering  \renewcommand{\arraystretch}{1.1}
  \caption{Error profile by the $C^1$-$P_6$ Argyris element on the  
   triangular meshes shown as in Figure \ref{t-grid}, for computing \eqref{s2}. }
  \label{t-8}
\begin{tabular}{c|cc|cc|cc}
\hline
grid & \multicolumn{2}{c|}{ $\| u-u_h\|_{0}$  \; $O(h^r)$} 
   & \multicolumn{2}{c|}{  $| u-u_h|_{1}$ \;$O(h^r)$} 
  &  \multicolumn{2}{c}{  $|u-u_h|_{a}$ \;$O(h^r)$} \\ \hline
    &  \multicolumn{6}{c}{ By the triangular $C^1$-$P_6$ Argyris element, $\mu_0=1$. }   \\
\hline  
 2&    0.432E-04 &  --- &    0.218E-03 &  --- &    0.172E-02 &  ---   \\
 3&    0.772E-07 &  9.13&    0.179E-05 &  6.93&    0.771E-04 &  4.48  \\
 4&    0.623E-09 &  6.95&    0.365E-07 &  5.61&    0.317E-05 &  4.60  \\
 5&    0.558E-11 &  6.80&    0.641E-09 &  5.83&    0.108E-06 &  4.88  \\
\hline 
 &  \multicolumn{6}{c}{ By the triangular $C^1$-$P_6$ Argyris element, $\mu_0=10$. }   \\
\hline 
 2&    0.176E-03 &  --- &    0.105E-02 &  --- &    0.274E-01 &  ---   \\
 3&    0.448E-05 &  5.30&    0.110E-03 &  3.26&    0.137E-01 &  1.00  \\
 4&    0.208E-05 &  1.11&    0.396E-04 &  1.47&    0.859E-02 &  0.67  \\
 5&    0.102E-05 &  1.03&    0.152E-04 &  1.39&    0.566E-02 &  0.60  \\
\hline   
    \end{tabular}%
\end{table}%

\end{document}